\documentclass[11pt]{amsart}
\usepackage[T1]{fontenc}
\usepackage[english]{babel}
\usepackage{kpfonts}

\usepackage{amsmath,amssymb,amsthm,enumerate,xspace}
\usepackage{accents}
\usepackage[colorlinks=true,citecolor=blue,linkcolor=blue]{hyperref}
\usepackage{mathrsfs}

\numberwithin{equation}{section}

\newtheorem{thm}{Theorem}[section]

\newtheorem{prop}[thm]{Proposition}
\newtheorem{lem}[thm]{Lemma}
\newtheorem{cor}[thm]{Corollary}

\newtheorem*{iproblem*}{Problem}

\theoremstyle{definition}

\newtheorem{defi}[thm]{Definition}
\newtheorem{rem}[thm]{Remark}

\newtheorem{exam}[thm]{Example}

\newcommand{\ZZ}{\mathbf{Z}}
\newcommand{\RR}{\mathbf{R}}

\newcommand{\PP}{\mathbf{P}}

\newcommand{\SL}{\mathbf{SL}}

\newcommand{\SSS}{\mathbf{S}}

\newcommand{\Aut}{{\rm Aut}}
\newcommand{\norma}{{\rm N}}

\newcommand{\se}{\subseteq}
\newcommand{\inv}{^{-1}}

\newcommand{\lra}{\longrightarrow}

\newcommand{\wt}{\widetilde}
\newcommand{\wh}{\widehat}

\newcommand{\ol}{\overline}
\newcommand{\bnda}{\Delta}
\newcommand{\bnd}{\partial}
\newcommand{\ext}{\mathrm{Ext}}
\newcommand{\prob}{\mathrm{P}}

\newcommand{\ru}{\mathrm{C}^\mathrm{b}_\mathrm{ru}}
\newcommand{\cont}{\mathrm{C}}
\newcommand{\cbu}{\mathrm{C}^\mathrm{b}_\mathrm{u}}
\newcommand{\cb}{\mathrm{C}^\mathrm{b}}

\def\hyph{-\penalty0\hskip0pt\relax}
\title[Furstenberg boundaries for pairs of groups]{Furstenberg boundaries for pairs of groups}
\date{February 2019}

\author[N. Monod]{Nicolas Monod}
\address{EPFL, Switzerland}
\email{nicolas.monod@epfl.ch}
\begin{document}
\begin{abstract}
Furstenberg has associated to every topological group $G$ a universal boundary $\bnd(G)$. If we consider in addition a subgroup $H<G$, the relative notion of $(G,H)$\hyph{}boundaries admits again a maximal object $\bnd(G,H)$. In the case of discrete groups, an equivalent notion was introduced by Bearden--Kalantar~\cite{Bearden-Kalantar_arx} as a very special instance of their constructions. However, the analogous universality does not always hold, even for discrete groups. On the other hand, it does hold in the affine reformulation in terms of convex compact sets, which admits a universal simplex $\bnda(G,H)$, namely the simplex of measures on $\bnd(G,H)$. We determine the boundary $\bnd(G,H)$ in a number of cases, highlighting properties that might appear unexpected.
\end{abstract}
\maketitle

\thispagestyle{empty}
\setcounter{tocdepth}{1}
\tableofcontents
\newpage

\section{Introduction}

A compact topological space on which a group $G$ acts by homeomorphisms is called a \textbf{$G$-flow}. When $G$ is a topological group, the action is assumed to be jointly continuous and $G$ Hausdorff.

Furstenberg~\cite{Furstenberg63,Furstenberg73_bnd} discovered the particular importance of the case where the action is \emph{minimal} and \emph{strongly proximal}; the flow is then called a \textbf{$G$\hyph{}boundary}. He showed for instance that each group $G$ admits a \emph{universal} boundary, now called the \textbf{Furstenberg boundary $\bnd(G)$ of $G$}.

Although $\bnd(G)$ is often a huge non-metrisable space, Furstenberg showed that for semi-simple Lie groups it reduces to a homogeneous space $\bnd(G) = G/P$, where $P$ is a minimal parabolic subgroup.

\medskip
It turns out, following Furstenberg and Glasner~\cite{Glasner_LNM}, that the notion of boundary is even more natural and transparent if we recast the whole discussion in the setting of \emph{convex compact spaces}:

An \textbf{affine $G$-flow} refers to a compact convex set $K$ endowed with a $G$-action preserving both the topology and the affine structure of $K$. Here $K$ is understood to lie in an arbitrary locally convex (Hausdorff) topological vector space over the reals. A \textbf{$G$\hyph{}morphism} is a $G$-equivariant continuous affine map.

Any $G$-flow $X$ gives an affine $G$-flow $\prob(X)$, the space of probability measures on $X$. There are of course many other affine $G$-flows. Now $X$ is a $G$\hyph{}boundary if and only if $\prob(X)$ satisfies just one single minimality condition: namely that it is \textbf{$G$\hyph{}irreducible}. This means by definition that it does not contain a proper affine subflow. The simplex of probability measures on $\bnd(G)$, which we denote by $\bnda(G)$, is universal in that setting.

\bigskip
This article considers the more general \emph{relative case}, where we are given a topological group $G$ together with a subgroup $H<G$. We shall see that there exist again canonical relative objects $\bnd(G,H)$ and $\bnda(G,H)$. However, there are interesting complications; notably, the topological flow $\bnd(G,H)$ behaves less well than its affine counterpart $\bnda(G,H)$. We therefore start off in the affine setting.

\begin{defi}
An \textbf{affine $(G,H)$-flow} is an affine $G$-flow with an $H$-fixed point. It is called \textbf{$(G,H)$\hyph{}irreducible} if it does not contain any smaller affine $(G,H)$-flow.
\end{defi}

The classical case corresponds to the trivial subgroup $H=1$.


\begin{prop}\label{prop:univ}
There exists a $(G,H)$\hyph{}irreducible affine flow that is \textbf{universal} in the sense that it admits a $G$\hyph{}morphism onto every $(G,H)$\hyph{}irreducible affine flow. Moreover, this universal flow is unique up to unique $G$\hyph{}morphisms.
\end{prop}

\begin{defi}
We denote the universal $(G,H)$\hyph{}irreducible affine flow by $\bnda(G, H)$.
\end{defi}

At this point we have a universal \emph{convex} compact object, but we lost sight of the initial discussion of $G$-flows: actions on arbitrary compact spaces. Not to worry: $\bnda(G, H)$ is in fact the simplex of probability measures $\prob(\bnd(G,H))$ of a flow $\bnd(G,H)$\,!

It is indeed well understood (since Bauer~\cite[Satz~13]{Bauer61}) when a convex compact space $K$ is of the form $\prob(X)$. This happens exactly when the set of extremal points $\ext(K)$ is closed and when moreover every point of $K$ is a \emph{unique} Choquet integral on $\ext(K)$. This realises $K$ as $\prob(\ext(K))$ and one calls $K$ a \textbf{Bauer simplex}~\cite[II.4.1]{Alfsen_book}.

\begin{thm}\label{thm:Bauer}
$\bnda(G, H)$ is a Bauer simplex.
\end{thm}

\begin{defi}
The \textbf{Furstenberg boundary $\bnd(G,H)$} of the pair $(G,H)$ is the set of extremal points $\bnd(G,H) = \ext\left(\bnda(G,H)\right)$.
\end{defi}

In other words, Theorem~\ref{thm:Bauer} states that there is a canonical $G$-flow, the Furstenberg boundary $\bnd(G,H)$, such that
\begin{equation*}
\prob\left(\bnd(G,H)\right) = \bnda(G, H).
\end{equation*}
Since we return to topological $G$-flows, we should define general $(G,H)$\hyph{}boundaries:

\begin{defi}
A $G$-flow $X$ is a \textbf{$(G,H)$\hyph{}boundary} if $\prob(X)$ is $(G,H)$\hyph{}irreducible.
\end{defi}

In fact this definition is equivalent to a characterisation given, in the case of discrete groups, by Bearden--Kalantar~\cite{Bearden-Kalantar_arx} in the context of their much more general non-commutative notion of Furstenberg--Hamana boundaries for unitary representations. To make this explicit, recall first that a probability measure $\mu$ on a $G$-flow $X$ is said to be \textbf{$G$-contracted} to a point $x$ if the Dirac mass $\delta_x$ belongs to the orbit closure $\ol{G \mu}$. Using the Krein--Milman theorem, one shows:

\begin{prop}\label{prop:traduc}
Let $X$ be a $G$-flow. Then $X$ is a $(G,H)$\hyph{}boundary if and only if $H$ fixes a measure in $\prob(X)$ and every such fixed measure is $G$-contracted to every point of $X$.
\end{prop}

This property coincides with the characterisation from~\cite{Bearden-Kalantar_arx}.

\medskip

Of course, we have not defined anything new when $H\lhd G$ is a normal subgroup: $(G,H)$-flows, boundaries and universal objects reduce to the classical objects for the quotient group $G/H$. However, general pairs $H<G$ exhibit completely different behaviours, as will become clear with a few examples which should serve as a warning, or better as an advertisement, for the new phenomena.

For instance, although the Furstenberg boundary $\bnd(G,H)$ is canonical, unique up to unique identification, and in a sense the \emph{maximal} $(G,H)$\hyph{}boundary, it is however not \emph{universal} in the strong sense of Proposition~\ref{prop:univ}.

\begin{prop}\label{prop:not:univ}
There is not always a $G$-map from $\bnd(G,H)$ to every $(G,H)$\hyph{}boundary, even for discrete groups.
\end{prop}

We can illustrate this very concretely on an example.

\begin{exam}\label{exam:F4}
Let $G=F_4$ be a free group on four generators and let $H<G$ be the first free factor $H=F_2$ in a splitting $G=F_2 * F_2$. Let $X$ be a topological circle. We endow $X$ with a $G$-action by specifying the following actions for each of the two copies of $F_2$.

For $H$, we choose two arbitrary rotations of the circle, at least one of which is non-trivial. Thus $H$ acts via a non-trivial abelian quotient.

For the second copy of $F_2$, we identify $X$ with the projective line and map $F_2$ to $\SL_2(\ZZ)$ by sending its generators to $\left(\begin{smallmatrix} 1&1\\ 0&1 \end{smallmatrix}\right)$ and $\left(\begin{smallmatrix} 1&0\\ 1&1 \end{smallmatrix}\right)$. This yields a projective action on $X$.

We claim that $X$ is a $(G,H)$\hyph{}boundary. Indeed, $H$ fixes the round measure on the circle. On the other hand, the second $F_2$ acts minimally and strongly proximally and hence \emph{any} measure can be contracted to any point. The claim follows.

\itshape
However, there is no $G$-map from $\bnd(G,H)$ to $X$. The key point is that $\bnd(G,H)$ will be shown to contain an $H$-fixed point. But, by construction, $X$ does not.\upshape
\end{exam}

\begin{rem}\label{rem:image:contains}
What remains true in general is that for any $(G,H)$\hyph{}boundary $X$ there is a $G$-map $\bnd(G,H) \to \prob(X)$ whose image \emph{contains} $X$, see Proposition~\ref{prop:image:contains}.
\end{rem}

A major advantage of considering non-discrete groups is that Furstenberg boundaries are completely understood for all connected Lie groups, where they are always homogeneous spaces~\cite[IV.3.3]{Glasner_LNM}. This again gets more complicated for the relative theory:

\begin{exam}\label{exam:affine}
Let $G= \RR^2 \rtimes \SL_2(\RR)$ be the special affine group of $\RR^2$ and let $H=\SL_2(\RR)$. Consider the visual compactification $D=\RR^2 \sqcup \SSS^1$ of $\RR^2$ obtained by gluing the circle of (oriented) directions; thus $D$ is a topological disc. We view $D$ as a $G$-flow where the action on the open disc is the affine action on $\RR^2$, while the action on $\SSS^1$ is induced by the linear representation on the space of directions via the quotient map to $H$.
\end{exam}

This example illustrates several interesting points, proved Section~\ref{sec:LC} and particularly Theorem~\ref{thm:af}:

\medskip
\begin{enumerate}[(1)]
\item $D$ is a $(G,H)$\hyph{}boundary.\label{af:pt:bnd}
\item $D$ is not a minimal $G$-flow. It will follow that the Furstenberg boundary $\bnd(G,H)$ is not minimal either --- much less homogeneous.\label{af:pt:no-min}
\item The Furstenberg boundary $\bnd(G)$ is realised by the natural $G$-action on the projective line $\PP^1$. Therefore, there is no $G$-map from $\bnd(G)$ to $D$.\\
We shall deduce that there is no $G$-map from $\bnd(G)$ to $\bnd(G,H)$.\label{af:pt:no-map}
\item $D$ is not maximal; in fact, the Furstenberg boundary $\bnd(G,H)$ consists of the open disc together with a non-metrisable compact space on which $\RR^2$ acts trivially, glued above $\SSS^1$.\label{af:pt:no-uni}
\end{enumerate}

\medskip
There are however also cases where the Furstenberg boundary of pairs behaves very simply and is indeed a quotient of the Furstenberg boundary of the ambient group:

\begin{exam}\label{exam:Gras}
Let $0<p<n$, let $G=\SL_n(\RR)$ and let $H<G$ be the block-wise upper-triangular subgroup with diagonal blocks of size $p$ and $n-p$. Then
\begin{equation*}
  \bnd(G,H) \cong G/H \cong \mathrm{Gr}_p(\RR^n),
\end{equation*}
the Grassmannian of $p$-spaces in $\RR^n$.
\end{exam}

This is a particular case of a general phenomenon for co\hyph{}compact subgroups $H<G$ in any topological group $G$. We shall prove that in this case $\bnd(G,H)$ is a homogeneous space $G/\wh{H}$ for a \emph{hull} $\wh{H}$ canonically attached to $H<G$ up to conjugation. In Example~\ref{exam:Gras}, we have $\wh{H}=H$, which holds more generally for parabolic subgroups:

\begin{thm}\label{thm:parabolic}
Let $G$ be a connected semi-simple Lie group with finite center and let $H$ be any parabolic subgroup.

Then $\bnd(G,H) \cong G/H$.

The corresponding statement holds for semi-simple algebraic groups over local fields.
\end{thm}

We already pointed out that the relative theory reduces to the classical one when $H<G$ is a normal subgroup. This invites the question: what happens at the opposite extreme, when $H$ is \textbf{malnormal} in $G$? That is, when $g H g\inv$ intersects $H$ trivially for all $g\notin H$.

\begin{thm}\label{thm:malnormal}
Let $H$ be a malnormal subgroup of a discrete group $G$. Suppose $H$ non-amenable.

Then $\bnd(G,H) \cong \beta(G/H)$, the Stone--\v{C}ech compactification of the $G$-set $G/H$.
\end{thm}

In a sense, this completely describes the malnormal case for discrete groups; indeed we shall see that when $H$ is amenable $\bnd(G,H)$ reduces again to the classical boundary $\bnd(G)$. For non-discrete groups, an example of a similar nature is given in Example~\ref {exam:Samuel}.

\begin{cor}\label{cor:free}
Let $G=H*H'$ be the free product of two discrete groups $H, H'$. If $H$ is non-amenable, then $\bnd(G,H) \cong \beta(G/H)$.\qed
\end{cor}

Indeed $H$ is malnormal in $G$, see~\cite[4.1.5]{Magnus-Karrass-Solitar}. This justifies the claim that the subgroup $H$ of Example~\ref{exam:F4} fixes a point in $\bnd(G,H)$.

\smallskip

We shall in fact prove a slightly stronger version of Theorem~\ref{thm:malnormal}, which has a consequence for the \textbf{wreath product} $J\wr H$ of two discrete groups. We recall that $J\wr H$ is the semi-direct product of the restricted product $\oplus_H J$ with $H$. This includes for instance the lamplighter group on $H$.

\begin{cor}\label{cor:wreath}
If $H$ is non-amenable, then $\bnd(J\wr H, H) \cong \beta\left(\oplus_H J\right)$.
\end{cor}

Another consequence of this method regards \emph{relatively hyperbolic} groups:

\begin{cor}\label{cor:relhyp}
Let $G$ be a discrete group that is hyperbolic relative to some family of subgroups $\{H_i\}_{i\in I}$.  Then $\bnd(G,H_i) \cong \beta(G/H_i)$ whenever $H_i$ is non-amenable.
\end{cor}

Boundary theory is inseparable from \textbf{amenability}: recall that the topological group $G$ is amenable if and only if its Furstenberg boundary $\bnd(G)$ is trivial. It is therefore not surprising that the relative Furstenberg boundary $\bnd(G, H)$ and its affine version $\bnda(G,H)$ are intimately connected to relative versions of amenability for pairs $H<G$.

There are two complementary such relative notions. First, we can ask when the relative Furstenberg boundary $\bnd(G, H)$ is trivial; the answer is straightforward and was already recorded in~\cite{Bearden-Kalantar_arx}  for discrete groups:

\begin{prop}\label{prop:coamen}
$\bnd(G, H)$ is trivial if and only if $H$ is co-amenable in $G$.
\end{prop}

The notion of co-amenability extends to arbitrary pairs the notion of amenability of the quotient group $G/H$ when $H\lhd G$ is a normal subgroup. For discrete groups, it is equivalent to the existence of a $G$-invariant mean on $G/H$, and more generally in the locally compact case to the weak containment of the trivial $G$-representation in the quasi-regular representation on $L^2(G/\ol{H})$, see~\cite{Eymard72}. For arbitrary topological groups it is defined by requiring that every affine $(G, H)$-flow has a $G$-fixed point and the above proposition is expected. Nonetheless there are some subtleties, such as the following direct consequence of a construction in~\cite{Monod-Popa} or~\cite{Pestov}.

\begin{prop}\label{prop:MP}
There is a discrete group $G$ with subgroups $H_1 < H_2 < G$ such that $\bnd(G, H_1)$ and $\bnd(G, H_2)$ are trivial but not $\bnd(H_2, H_1)$.

Moreover we can choose $H_1$ normal in $H_2$ and  $H_2$ normal in $G$.
\end{prop}

At the opposite end from co-amenability, we get the other relative notion by asking when $\bnd(G,H)$ is isomorphic (as $G$-flow) to $\bnd(G)$.

\begin{prop}\label{prop:relamen}
$\bnd(G, H) \cong \bnd(G) $ if and only if $H$ is amenable relative to $G$.
\end{prop}

We recall that $H$ is called \textbf{amenable relative to $G$} if every affine $G$-flow has an $H$-fixed point. Although it is clearly complementary to co-amenability, this notion comes with a surprise. For discrete groups, it simply amounts to the amenability of the group $H$. But already in the locally compact setting, it is a priori weaker than the amenability of $H$, see~\cite{Caprace-Monod_rel}. It remains an open problem to exhibit a locally compact example where the notions do not coincide. Finally, beyond locally compact groups, there is a complete divergence. Indeed, any subgroup of an amenable group will be relatively amenable, but need not be amenable. The first example is from 1973, when la Harpe~\cite{Harpe73} showed that the unitary group of the separable Hilbert space in the strong operator topology is an amenable Polish group, whereas it contains any countable group as a discrete closed subgroup, including the non-amenable ones, as witnessed by the regular representation.

\medskip

In fact these two relative notions are just extreme cases of a very basic pre-order relation on the set of all subgroups of a given group $G$; this definition is taken from~\cite[\S2.3]{Portmann_PhD} and~\cite[\S7.C]{Caprace-Monod_rel}.

\begin{defi}\label{defi:relco}
Let $H, H'$ be subgroups of a topological group $G$. We say that \textbf{$H$ is co-amenable to $H'$ relative to $G$} if every affine $(G,H)$-flow has an $H'$-fixed point.
\end{defi}

We see that a co-amenable subgroup $H$ corresponds to the special case $H=G$, whilst a relatively amenable subgroup $H'$ corresponds to $H=1$.

Even for discrete groups, this notion has clarifying virtues. For instance, the situation of Proposition~\ref{prop:MP} can be rephrased by saying that $H_1 \lhd H_2$ is co-amenable to $H_2$ relative to $G$, even though it is no co-amenable \emph{in} $H_2$.

\section{Irreducible affine flows}
Let $G$ be a topological group and $H<G$ a subgroup.

We first comment on the fact that an affine $G$-flow $K$ and $G$\hyph{}morphisms were defined intrinsically on $K$ rather than on some locally convex topological vector space containing $K$. There is no loss of generality in assuming that the $G$-action on $K$ actually comes from a representation by continuous linear operators on the ambient space, although this might require us to modify that ambient space without changing $K$. Indeed we can embed $K$ in the state space on the continuous affine functions on $K$. However, our focus will always be on $K$ only. A similar remark applies to morphisms.

\smallskip
A straightforward compactness argument with Zorn's lemma and the fact that fixed points constitute a closed convex subset gives the following.

\begin{lem}\label{lem:exists:min}
Every affine $(G,H)$-flow contains a $(G,H)$\hyph{}irreducible one.\qed
\end{lem}

On the other hand, the definitions imply:

\begin{lem}\label{lem:onto}
Every $G$\hyph{}morphism from an affine $(G,H)$-flow to a $(G,H)$\hyph{}irreducible one is onto.\qed
\end{lem}

Given a convex compact set $K$, we denote the set of its \textbf{extremal points} by $\ext(K)$. We recall that $\ext(K)$ can variously be closed, or dense~\cite{Poulsen}, or non-Borel~\cite[\S VII]{Bishop-deLeeuw}.

The following reformulation is essentially an exercice around the Krein--Milman theorem.

\begin{lem}\label{lem:KM}
Let $K$ be an affine $(G,H)$-flow. Then
\begin{equation*}
\text{$K$ is $(G,H)$\hyph{}irreducible} \kern2mm\Longleftrightarrow \kern2mm \forall x\in K^H : \kern2mm \ext(K) \se \ol{G x}.
\end{equation*}
\end{lem}

\begin{proof}
The condition on the right is sufficient. Indeed, if $L\se K$ is an affine $(G,H)$-flow, it contains an $H$-fixed point and hence contains $\ext(K)$. Thus $L=K$ by Krein--Milman.

Conversely, suppose that $K$ is $(G,H)$\hyph{}irreducible and let $x\in K^H$. Then the closed convex hull of $\ol{G x}$ is $K$ and hence $\ol{G x}$ contains $\ext(K)$ by Milman's partial converse to Krein--Milman~\cite[V.8.5]{Dunford-Schwartz_I}.
\end{proof}

\begin{lem}\label{lem:coincidence}
Let $K,L$ be affine $(G,H)$-flows and let $f, f'\colon K\to L$ be $G$\hyph{}morphisms. Consider the coincidence set $K_0=\{x\in K : f(x) = f'(x)\}$.

If $L$ is $(G,H)$\hyph{}irreducible, then $f(K_0)=L$.
\end{lem}

The point here is that we do not know a priori that $f(K_0)$ contains an $H$-fixed point (or even is non-empty).

\begin{proof}
  By Krein--Milman, it suffices to prove that every extremal point $\zeta$ of $L$ belongs to $f(K_0)$. Let $x\in K^H$. Then $f(x)$ and $f'(x)$ belong to $L^H$; hence so does $z= (f(x) + f'(x))/2$. By Lemma~\ref{lem:KM}, there is a net $(g_\alpha)_{\alpha\in A}$ in $G$ such that $g_\alpha z \to \zeta$. Upon replacing $(g_\alpha)$ by a subnet, $g_\alpha x$ converges to some $y\in K$. Now $\zeta$ is the limit of
  \begin{equation*}
    g_a z = g_a \tfrac12 (f(x) + f'(x)) =  \tfrac12 (f(g_a  x) + f'(g_a  x)) \to   \tfrac12 (f(y) + f'(y) ) .
  \end{equation*}
Since $\zeta$ is extremal, we deduce that $f(y)=f'(y)=\zeta$, which witnesses that $\zeta\in f(K_0)$.
\end{proof}

If we apply Lemma~\ref{lem:coincidence} to $K=L$ and $f$ the identity, we deduce the following.

\begin{cor}\label{cor:no:endo}
Let $K$ be a $(G,H)$\hyph{}irreducible affine flow. Then the identity is the only $G$\hyph{}morphism $K\to K$.\qed
\end{cor}

We now establish the existence of the universal $(G,H)$\hyph{}irreducible affine flow $\bnda(G,H)$.

\begin{proof}[Proof of Proposition~\ref{prop:univ}]
Let $\{K_i\}_{i\in I}$ a family of isomorphism representatives of all $(G,H)$\hyph{}irreducible affine flows. Such a set indeed exists, see Remark~\ref{rem:cardi} below. The product $\prod_{i\in I} K_i$ has $H$-fixed points and hence it contains an $(G,H)$\hyph{}irreducible affine subflow $K$ by Lemma~\ref{lem:exists:min}. The coordinate projections provide $G$\hyph{}morphisms to all $(G,H)$\hyph{}irreducible affine flows. These morphisms are onto by Lemma~\ref{lem:onto}.

Suppose now that $K'$ is another $(G,H)$\hyph{}irreducible affine flow with this property. Then we have $G$\hyph{}morphisms $K\to K'$ and $K'\to K$. Applying Corollary~\ref{cor:no:endo} to the composition of these morphisms in the two possible orders shows that they are isomorphisms, and then that they are unique. This completes the proof of Proposition~\ref{prop:univ}.
\end{proof}

\begin{rem}\label{rem:cardi}
  Given $G$, we can bound the cardinal of an arbitrary $(G,H)$\hyph{}irreducible affine flow $K$ using basic cardinal functions from general topology such as \emph{density} and \emph{weight}, for which we refer to~\cite[a-3]{Hart-Nagata-Vaughan}. For instance, the cardinal of $K$ is bounded by the exponential of its weight, which is bounded by the exponential of its density. The density of $K$ is bounded by  the density of $\ext(K)$ (as soon as the latter is infinite), which is bounded by the density of $G$ in view of Lemma~\ref{lem:KM}. Now that the cardinal of $K$ is bounded, we also have a bound on all possible structures of affine $G$-flow on $K$.
\end{rem}

\begin{rem}\label{rem:univ}
Combining universality with Lemma~\ref{lem:exists:min}, we also see that $\bnda(G, H)$ admits a $G$\hyph{}morphism to every affine $(G,H)$-flow. 
\end{rem}

A standard observation is that whenever a $G$-object is unique up to \emph{unique isomorphisms}, it automatically inherits an action of the automorphism group $\Aut(G)$. This fact was used by Furstenberg for the boundary $\bnd(G)$, see~\cite[II.4.3]{Glasner_LNM}.

In the present case, this holds for the automorphisms that preserve the subgroup $H$. We denote by $\Aut(G,H)<\Aut(G)$ the group of these automorphisms. Notice that the pre-image of $\Aut(G,H)$ in $G$ under the conjugation homomorphism $G\to\Aut(G)$ is precisely the normaliser $\norma_G(H)$ of $H$ in $G$.

\begin{cor}\label{cor:bnda:extends}
The automorphism group $\Aut(G,H)$ has an affine action on $\bnda(G, H)$ such that the resulting action of $\norma_G(H)$ induced by $G\to\Aut(G)$ coincides with the one given by the original $G$-action on $\bnda(G, H)$.\qed
\end{cor}

Indeed, given $\alpha\in \Aut(G)$, we have a new $G$-action on $\bnda(G, H)$ given by $\alpha$-twisting the original action. This turns $\bnda(G, H)$ into a universal affine $(G,H)$-flow if $\alpha\in \Aut(G,H)$. The unique isomorphism identifying this new $(G,H)$-flow $\bnda(G, H)$ with the original one defines the action of $\alpha$. The argument in~\cite[II.4.3]{Glasner_LNM} can be copied verbatim.

\begin{rem}\label{rem:morph:pairs}
We can also consider more generally morphisms of pairs of groups $H<G$ and $H'<G'$. Let thus $f\colon G' \to G$ be a continuous group homomorphism such that $f(H')\se H$. Then every affine $(G,H)$-flow becomes an affine $(G',H')$-flow by pull-back. By Remark~\ref{rem:univ}, there exists an $f$-equivariant morphism $\bnda(G', H')\to \bnda(G, H)$. We shall see in Remark~\ref{rem:no-morph:pairs} below that the corresponding fact does not hold for $\bnd(G', H')$ and $\bnd(G, H)$.
\end{rem}

\section{Back to topological flows}
By definition, the functor $X\mapsto \prob(X)$ sends $G$-flows to affine $G$-flows and $(G,H)$\hyph{}boundaries to $(G,H)$\hyph{}irreducible affine flows. What is less clear is to which extent one can go in the other direction. Specifically, it is not clear how to obtain a $(G,H)$\hyph{}boundary from a general $(G,H)$\hyph{}irreducible affine flow. This contrasts with the classical case, where the closure $\ol{\ext(K)}$ is a $G$\hyph{}boundary for every $G$\hyph{}irreducible affine flow~\cite[III.2.3]{Glasner_LNM}.

Nonetheless, much more is true for the universal $(G,H)$\hyph{}irreducible affine flow $\bnda(G,H)$, since we will prove that it is a Bauer simplex. 

\medskip
Before starting, we recall that one can always homeomorphically identify $X$ with its image in $\prob(X)$. Moreover, if $Y$ is a closed subspace of the compact space $X$, then the Tietze--Urysohn extension theorem realizes $\prob(Y)$ as a convex compact subspace of $\prob(X)$. Finally, recall that for every convex compact $K$ there is a \textbf{barycentre map} $\beta\colon \prob(K) \to K$ given by integrating the measures.

\begin{proof}[Proof of Theorem~\ref{thm:Bauer}]
Notice that $\prob(\bnda(G,H))$ is an affine $(G,H)$-flow since $\bnda(G,H)$ and hence $\prob(\bnda(G,H))$ contains an $H$-fixed point. Thus, applying Lemma~\ref{lem:exists:min}, there is a $(G,H)$\hyph{}irreducible subflow $L$ in $\prob(\bnda(G,H))$. Now the barycentre map $\beta$ associated to $\bnda(G,H)$ induces a $G$\hyph{}morphism $L\to \bnda(G,H)$. This is an isomorphism by the universality of $\bnda(G,H)$.

Let now $x$ be any extremal point of $\bnda(G,H)$. Then the only measure in $\prob(\bnda(G,H))$ mapped to $x$ is the Dirac mass $\delta_x$, see~\cite[1.4]{Phelps_LNM}. Therefore $\delta_x \in L$ for any $x\in\ext(\bnda(G,H))$. Since $L$ is closed, it contains $\delta_x$ for all $x\in\ol{\ext(\bnda(G,H))}$. Since $L$ is also convex, it contains $\prob\left(\ol{\ext(\bnda(G,H))}\right)$ viewed as a subspace of $\prob(\bnda(G,H))$. By irreducibility, $L$ is exactly $\prob\left(\ol{\ext(\bnda(G,H))}\right)$.

The fact that the barycentre map restricts to an isomorphism
\begin{equation*}
\prob\left(\ol{\ext(\bnda(G,H))}\right) \lra\bnda(G,H)
\end{equation*}
precisely means that the set $\bnd(G,H)$ defined as $\ext(\bnda(G,H))$ is closed and that $\bnda(G,H)$ identifies with $\prob(\bnd(G,H))$, see~\cite[II.4.1]{Alfsen_book}.
\end{proof}

Given that the Furstenberg boundary $\bnd(G, H)$ is canonically defined in terms of $\bnda(G, H)$, Corollary~\ref{cor:bnda:extends} implies:

\begin{cor}\label{cor:bnd:extends}
The automorphism group $\Aut(G,H)$ has an action by homeomorphisms on $\bnd(G, H)$ such that the resulting action of $\norma_G(H)$ induced by $G\to\Aut(G)$ coincides with the one given by the original $G$-action on $\bnd(G, H)$.\qed
\end{cor}

Likewise, Corollary~\ref{cor:no:endo} implies the corresponding statement for $(G,H)$\hyph{}boundaries thanks to the functor $X\mapsto \prob(X)$ :

\begin{cor}\label{cor:no:endo:top}
Let $X$ be a $(G,H)$\hyph{}boundary. Then the identity is the only continuous $G$-map $X\to X$.\qed
\end{cor}

We now justify Remark~\ref{rem:image:contains}:

\begin{prop}\label{prop:image:contains}
For any $(G,H)$\hyph{}boundary $X$ there is a continuous $G$-map from $\bnd(G,H)$ to $P(X)$ whose image contains $X$ --- when $X$ is viewed as a subspace of $\prob(X)$.
\end{prop}

\begin{proof}
Since $X$ is a $(G,H)$\hyph{}boundary, there is a $G$\hyph{}morphism from $\bnda(G,H)$ onto $P(X)$. Now the proposition is a consequence of the following observation. Given any surjective continuous affine map $\pi\colon K\to L$ between convex compact sets, $\pi(\ext(K))$ contains $\ext(L)$. Indeed, let $x$ be an extremal point of $L$. Then the fiber $\pi\inv(\{x\})$ is a convex compact subset of $K$. One checks that any extremal point of $\pi\inv(\{x\})$ is extremal in $K$, whence the observation. 
\end{proof}

\section{Relative co-amenability}
Let $G$ be a topological group. As indicated in the Introduction, the amenability of $G$ and of the relative position of its subgroups $H, H'$ are all subsumed in one simple notion: $H$ is called co-amenable to $H'$ relative to $G$ if every affine $(G,H)$-flow has an $H'$-fixed point. It is natural that this can be rephrased using the universal affine spaces:

\begin{prop}\label{prop:relco}
Let $H, H'$ be subgroups of $G$. The following are equivalent:
\begin{enumerate}[(i)]
\item $H$ is co-amenable to $H'$ relative to $G$.\label{pt:relco:def}
\item There exists a $G$\hyph{}morphism $\bnda(G, H') \to \bnda(G, H)$. \label{pt:relco:map}
\item $H'$ fixes a point in $\bnda(G, H)$. \label{pt:relco:fp}
\end{enumerate}
\end{prop}

\begin{proof}
\eqref{pt:relco:def}$\Rightarrow$\eqref{pt:relco:fp} holds by definition and \eqref{pt:relco:fp}$\Rightarrow$\eqref{pt:relco:map} follows from Remark~\ref{rem:univ} applied to $H'$.

Assume now~\eqref{pt:relco:map} and let $K$ be any affine $(G,H)$-flow. Composing the $G$\hyph{}morphism $\bnda(G, H') \to \bnda(G, H)$ with the $G$\hyph{}morphism $\bnda(G, H)\to K$ of Remark~\ref{rem:univ}, we see that $K$ has an $H'$-fixed point, whence~\eqref{pt:relco:def}.
\end{proof}

Combining Proposition~\ref{prop:relco} with the fact that $\bnda(G, H')$ and $\bnda(G, H)$ have no non-trivial endomorphisms (Corollary~\ref{cor:no:endo}), we deduce a characterisation that also holds with the Furstenberg boundary $\bnd(G, H)$ instead of $\bnda(G, H)$:

\begin{cor}\label{cor:same:bnd}
Let $H, H'$ be subgroups of $G$. The following are equivalent:
\begin{enumerate}[(i)]
\item $H,H'$ are co-amenable to each other relative to $G$.
\item $\bnda(G, H)$ and $\bnda(G, H')$ are isomorphic affine $G$-flows.
\item The Furstenberg boundaries $\bnd(G, H)$ and $\bnd(G, H')$ are isomorphic $G$-spaces.
\end{enumerate}\qed
%
\end{cor}

For nested subgroups, we deduce:

\begin{cor}\label{cor:nested}
Let $H_1 < H_2 < G$. Then
\begin{equation*}
\bnd(G, H_1) \cong \bnd(G, H_2) \kern2mm\Longleftrightarrow \kern2mm \text{$H_1$ is co-amenable to $H_2$ relative to $G$.}
\end{equation*}
This holds in particular if $H_1$ is co-amenable in $H_2$.\qed
\end{cor}

The fact that the co-amenability of $H_1$ in $H_2$ is not necessary is illustrated by the examples in~\cite{Monod-Popa} or~\cite{Pestov}, where one can assume $H_1\lhd H_2\lhd G$. This, together with Corollary~\ref{cor:nested}, justifies Proposition~\ref{prop:MP}.

Notice that Corollary~\ref{cor:nested} contains the following particular cases, which were established in~\cite{Bearden-Kalantar_arx} for discrete groups:

\begin{cor}
Let $H$ be a subgroup of the topological group $G$. Then
\begin{align*}
\bnd(G, H) \cong \bnd(G) \kern2mm&\Longleftrightarrow \kern2mm \text{$H$ amenable relative to $G$,}\\
\text{$\bnd(G, H)$ is trivial} \kern2mm&\Longleftrightarrow \kern2mm \text{$H$ is co-amenable in $G$.}
\end{align*}
\qed
\end{cor}

Using the characterisation~\eqref{pt:relco:fp} of Proposition~\ref{prop:relco}, the following reduces to a compactness argument together with the continuity of the action on $\bnda(G, H)$.

\begin{cor}
Let $H, H'$ be subgroups of $G$. There is a subgroup $H''$ of $G$ containing $H'$ which is maximal amongst all subgroups to which $H$ is co-amenable relative to $G$.

Moreover, $H''$ is closed in $G$.\qed
\end{cor}

In the case $H=H'$, we can further apply Corollary~\ref{cor:same:bnd} and deduce:

\begin{cor}\label{cor:hull}
Any subgroup $H<G$ is contained in a closed subgroup $\wh{H}<G$ which is maximal amongst all subgroups to which $H$ is co-amenable relative to $G$. Moreover, we have $\bnd(G, H) =\bnd(G, \wh{H})$.\qed
\end{cor}

We shall refer to $\wh{H}$ as a \textbf{hull} of $H$ in $G$. For instance, when $H$ is trivial, $\wh{H}$ is just a maximal relatively amenable subgroup of $G$. In the discrete case, or for semi-simple groups, this coincides with maximal amenable subgroups.

\section{Co-compact subgroups}
The key reason why the Furstenberg boundary $\bnd(G)$ of a semi-simple Lie group $G$ is just a homogeneous space is that $G$ contains a co\hyph{}compact amenable subgroup. In the relative case, we can benefit from the co\hyph{}compactness of certain subgroups even when they are not amenable. Specifically, let $H<G$ be a subgroup and consider a hull $H< \wh{H} < G$ in the sense of Corollary~\ref{cor:hull}.

\begin{thm}\label{thm:coc:hull}
If $G/\wh{H}$ is compact (for instance if $H$ itself is co\hyph{}compact), then
\begin{equation*}
\bnd(G, H) =\bnd(G, \wh{H}) \cong G/\wh{H}.
\end{equation*}
Moreover in that case $\wh{H} < G$ is unique up to conjugacy amongst co\hyph{}compact hulls of $H$.
\end{thm}

The classical case is again contained here by setting $H=1$. However, in general, we caution that it is important to take $\wh{H}$ containing $H$. Indeed, the group $G$ of Example~\ref{exam:affine} contains a co\hyph{}compact amenable subgroup $H'$, but $\bnd(G, H)$ is far from homogeneous even though $H$ is trivially co-amenable to $H'$ relative to $G$.

\begin{proof}[Proof of Theorem~\ref{thm:coc:hull}]
We know already $\bnd(G, H) =\bnd(G, \wh{H})$ from Corollary~\ref{cor:hull}. By definition and Proposition~\ref{prop:relco}, $\wh{H}$ fixes a point $x$ in $\bnda(G/\wh{H})$. The orbit $Gx$ is a quotient of $G/\wh{H}$ and hence is closed in $\bnda(G/\wh{H})$ because $\wh{H}$ is co\hyph{}compact. On the other hand, $x$ is $H$-fixed and therefore Lemma~\ref{lem:KM}, implies that $Gx$ contains the extremal points $\bnd(G, \wh{H})$ of $\bnda(G/\wh{H})$. This implies $G x = \bnd(G, \wh{H})$. On the other hand, the stabiliser of $x$ cannot be larger than $\wh{H}$ by maximality and Proposition~\ref{prop:relco} again. The uniqueness of $\wh{H}$ up to conjugacy now follows from the uniqueness of $\bnd(G, H)$.
\end{proof}

Suppose that a topological group $G$ contains a co\hyph{}compact amenable subgroup $P$. This is for instance the case of all connected locally compact groups~\cite[3.3]{Anantharaman02} and algebraic groups of local fields. More generally we allow $P$ to be amenable relative to $G$, although this makes no difference in the two classes of examples just mentioned, see~\cite{Caprace-Monod_rel}.

Then $P$ is contained in a \emph{maximal} relatively amenable subgroup, which is just $\wh{P}$; the latter is still co\hyph{}compact of course. In the case of semi-simple Lie groups or algebraic groups over local fields, $\wh{P}$ is a minimal parabolic subgroup. This is the motivation for the following statement.

\begin{thm}\label{thm:parabolic:abstract}
Let $G$ be a topological group admitting a co\hyph{}compact relatively amenable subgroup $P$.

For any closed subgroup $H$ containing $\wh{P}$, we have $\bnd(G, H) = G/H$.
\end{thm}

In particular this implies Theorem~\ref{thm:parabolic}.

\begin{proof}[Proof of Theorem~\ref{thm:parabolic:abstract}]
In view of Theorem~\ref{thm:coc:hull}, we know already $\bnd(G, H) =\bnd(G, \wh{H}) = G/\wh{H}$. It remains to prove $\wh{H}=H$. It is known since Furstenberg that $G/\wh{P}$ is the Furstenberg boundary of $G$; we can see this as a special case of Theorem~\ref{thm:coc:hull}. In particular, $G/H$ is a $G$\hyph{}boundary and hence a $(G,H)$\hyph{}boundary. By Proposition~\ref{prop:image:contains}, we have a continuous $G$-map $G/\wh{H}\to \prob(G/H)$ whose image contains $G/H$. Since $G/\wh{H}$ is homogeneous, this means that we have in fact a $G$-map $G/\wh{H}\to G/H$. It follows that $H$ contains $\wh{H}$ and hence $\wh{H}=H$.
\end{proof}

Here is an explicit illustration of the interplay between Theorem~\ref{thm:coc:hull} and Theorem~\ref{thm:parabolic:abstract}:

\begin{exam}
Let $G=\SL_3(\RR)$ and let $H$ be the (non-cocompact!) subgroup of matrices of the form form $\left(\begin{smallmatrix} * & * & *\\ * & * & *\\ 0 & 0 & 1\end{smallmatrix}\right)$. Thus $H$ is isomorphic to the special affine group of $\RR^2$ studied more closely in Section~\ref{sec:LC}. Let moreover $Q$ be the parabolic subgroup $\left(\begin{smallmatrix} * & * & *\\ * & * & *\\ 0 & 0 & *\end{smallmatrix}\right)$. Then we have
\begin{equation*}
\wh{H}= Q, \kern2mm \bnd(G, H) = \bnd(G, Q)  = G/Q = \mathrm{Gr}_2(\RR^3)
\end{equation*}
wherein the first equality is defined up to conjugation. Indeed, if we justify $\wh{H}= Q$, the remaining identifications hold by Theorems~\ref{thm:coc:hull} and~\ref{thm:parabolic:abstract}.

Notice first that $H$ is co-amenable in $Q$. Therefore $H$ has a hull $\wh{H}$ containing $Q$. On the other hand, $\wh{Q}=Q$, as established in Theorem~\ref{thm:parabolic:abstract}. But $H<Q<\wh{H}$ implies that $Q$ is co-amenable in $\wh{H}$ and hence $\wh{Q}$ contains $\wh{H}$; the claim follows.
\end{exam}

We note that in this example $Q$ accidentally happens to be a maximal subgroup of $G$, but the reasoning used above applies beyond this special case.

\section{Malnormal subgroups}
Recall that a subgroup $H<G$ of a group $G$ is called \textbf{almost malnormal} if the intersection $H\cap g H g\inv$ is finite for all $g\notin H$. Theorem~\ref{thm:malnormal} from the Introduction holds in this wider setting, indeed in an even more general one:

\begin{thm}\label{thm:a:malnormal}
Let $H$ be a subgroup of a discrete group $G$. Suppose that $H$ is non-amenable but that $H\cap g H g\inv$ is amenable for all $g\notin H$.

Then $\bnd(G,H) \cong \beta(G/H)$.
\end{thm}

\begin{proof}
We first claim that the Stone--\v{C}ech compactification $\beta(G/H)$ of the $G$-set $G/H$ is a $(G,H)$\hyph{}boundary, using the characterisation of Proposition~\ref{prop:traduc}. There is an $H$-invariant measure, namely the Dirac mass at the trivial coset in $G/H$. We write $*$ for this coset. That measure can be $G$-contracted to any point of $\beta(G/H)$ since it is defined by a point whose orbit is dense. It therefore suffices to show that there is no other $H$-invariant measure on $\beta(G/H)$. To this end, we recall that probability measures on $\beta(G/H)$ correspond to means (i.e.\ finitely additive measures) on $G/H$. Indeed, by the definition of the Stone--\v{C}ech compactification, there is a natural identification between the algebras $\cont(\beta(G/H))$ and $\ell^\infty(G/H)$. Therefore we want to show that the Dirac mass at the trivial coset $*$ is the only $H$-invariant mean on $G/H$. Equivalently, that there is no $H$-invariant mean on the $H$-set $S=G/H\setminus\{*\}$. Our assumption is equivalent to the amenability of the stabiliser in $H$ of any point in $S$. Therefore, the existence of an $H$-invariant mean on $S$ would contradict the non-amenability of $H$ (see e.g.\ Lemma~4.5 in~\cite{Glasner-Monod}). The claim is proven.

It remains to verify that $\prob(\beta(G/H))$ is also universal. Let thus $K$ be any $(G,H)$\hyph{}irreducible affine flow. Pick an $H$-fixed point in $K$; the associated orbital map yields a $G$-map $G/H\to K$. By the universal property of the Stone--\v{C}ech compactification, this extends to a continuous $G$-map $\beta(G/H)\to K$. The latter induces a $G$\hyph{}morphism from $\prob(\beta(G/H))$ to $\prob(K)$ which gives the desired $G$\hyph{}morphism to $K$ when composed with the barycentre map on $K$.
\end{proof}

\begin{proof}[Proof of Corollary~\ref{cor:wreath}]
Let $G=J\wr H$ with $H$ non-amenable. In order to be able to apply Theorem~\ref{thm:a:malnormal}, it suffices to check that $H$ is almost malnormal in $G$, which is equivalent to checking that the stabiliser in $H$ of a \emph{non-trivial} element of $\oplus_H J$ is finite. This stabiliser must in particular preserve the support of this element, which is a non-empty subset of $H$. The definition of the restricted product $\oplus_H J$ shows that this subset is finite and hence its stabiliser in $H$ is finite too.
\end{proof}

\begin{proof}[Proof of Corollary~\ref{cor:relhyp}]
It was proved by Osin that every $H_i$ is almost malnormal, see Theorem~1.4(2) in~\cite{Osin06_AMS}. Therefore, we can again apply Theorem~\ref{thm:a:malnormal} above.
\end{proof}

\section{The facts in the case of the affine group}
\label{sec:LC}
This section is devoted to the study of the special affine group of $\RR^2$. All the arguments hold mutatis mutandis also for $\RR^n$ with $n\geq 3$, and in fact even simplify slightly in what regards the uniqueness of a certain finitely additive measure below, thanks to Kazhdan's property~(T).


Let thus
\begin{equation*}
G= \RR^2 \rtimes H \kern2mm \text{with} \kern2mm H=\SL_2(\RR)
\end{equation*}
and let $D$ be a topological disc. Let $G$ act on the boundary circle $\SSS^1$ of $D$ by the action on the space of directions. Thus this is a double cover of the projective $G$-action on $\PP^1$. As for the interior $\mathring D$ of the disc, we identify it with $\RR^2$ endowed with the affine $G$-action. Notice that the $G$-actions on $\mathring D$ and on $\SSS^1$ do indeed combine to turn $D$ into a $G$-flow.

\begin{prop}\label{prop:D:bnd}
$D$ is a $(G,H)$\hyph{}boundary.
\end{prop}

\begin{proof}
The group $H$ has an invariant probability measure on $D$ since it fixes a point in $\RR^2\cong \mathring D$, the origin. Moreover this measure can be $G$-contracted to any point of $D$ since the orbit of that point is dense. It suffices therefore to show that $H$ does not admit any other invariant probability measure $\mu$ on $D$.

First, $\mu$ gives mass zero to $\SSS^1$. Indeed, even the quotient $\PP^1$ has no $H$-invariant measure --- in fact this quotient is well known to be the Furstenberg boundary of $H$. Next, on $\mathring D \cong \RR^2$, we have an even stronger fact, which we will need again later: the only $H$-invariant \emph{finitely additive} probability measure on the algebra of Borel subsets of $\RR^2$ is the Dirac mass at the origin, see e.g.~\cite[ch.~2]{Harpe-Valette}.
\end{proof}

Notice that $D$ has two $G$-orbits; therefore it is not a minimal $G$-flow. Proposition~\ref{prop:image:contains} now implies that $\bnd(G,H)$ is not minimal either. Thus we have already established points~\eqref{af:pt:bnd} and~\eqref{af:pt:no-min} of Example~\ref{exam:affine}.

We now proceed to a somewhat non-explicit characterisation of the Furstenberg boundary $\bnd(G,H)$. Denote by $\ru(G)$ the space of right uniformly continuous bounded function on $G$. Recall that the right uniform continuity is equivalent to the continuity in sup-norm of the \emph{left} regular representation. We further denote by $\ru(G/H)$ the subspace of those functions that are right $H$-invariant. Although these functions will be seen as continuous functions on $\RR^2$, we caution that the uniform continuity requirement here is much stronger than the uniform continuity with respect to $\RR^2$. In fact, using the semi-direct product structure of $G$, we have:

\begin{lem}\label{lem:ru}
A bounded uniformly continuous function $f$ on $\RR^2$ is in $\ru(G/H)$ if and only if
\begin{equation*}
\lim_{n\to\infty}\sup_{v\in\RR^2} \left| f(h_n v) - f(v)\right| =0
\end{equation*}
for every sequence $h_n\to 1$ in $H$.\qed
\end{lem}

In any case, the restriction map $\ru(G/H)\to \cb(\RR^2)$ realises $\ru(G/H)$ as a closed subalgebra of $\cb(\RR^2)$ and therefore determines a compactification $\alpha(\RR^2)$ which has a structure of $G$-flow. The continuity of the action follows from the definition of $\ru(G/H)$.

\begin{thm}\label{thm:af}
\leavevmode
\begin{enumerate}[(i)]
\item The Furstenberg boundary $\bnd(G,H)$ is isomorphic to $\alpha(\RR^2)$.\label{thm:af:pt:isom}
\item The $\RR^2$-action on the corona $\alpha(\RR^2)\setminus \RR^2$ is trivial.\label{thm:af:pt:triv}
\item The identification $\RR^2\cong\mathring D$ extends to a $G$-map $\alpha(\RR^2)\to D$ mapping the corona onto $\SSS^1$.\label{thm:af:pt:quot}
\item The compact space $\alpha(\RR^2)$ is non-metrisable.\label{thm:af:pt:huge}
\end{enumerate}
\end{thm}

The point of~\eqref{thm:af:pt:quot} is that even when we know that  $\alpha(\RR^2)$ is $\bnd(G,H)$, the corona is a priori only mapped to $\prob(\SSS^1)$.

\begin{proof}[Proof of Theorem~\ref{thm:af}]
We begin with~\eqref{thm:af:pt:triv}. Let $(p_i)_{i\in I}$ be any net tending to infinity in the locally compact space $\RR^2$ (this means by definition that $p_i$ eventually leaves any compact subset). Write $p_i=\left(\begin{smallmatrix} x_i\\y_i\end{smallmatrix}\right)$ and suppose that it converges in $\alpha(\RR^2)$. By symmetry under rotations and scaling, it suffices to show that $p_i$ and $p'_i=\left(\begin{smallmatrix} x_i+1\\y_i\end{smallmatrix}\right)$ have the same limit. Suppose first that $|y_i|$ tends to infinity. Let $h_i\in H$ be the matrix $\left(\begin{smallmatrix} 1&1/y_i\\ 0&1 \end{smallmatrix}\right)$. Since $h_i\to 1$ and $h_i p_i = p'_i$, indeed  $p_i$ and $p'_i$ have the same limit. Therefore we can reduce to the case where $y_i$ remains bounded. Now $|x_i|$ tends to infinity and we take $a_i\in H$ to be the diagonal matrix with diagonal entries $(x_i +1)/x_i$ and $x_i/(x_i +1)$. Again, $a_i\to 1$. This time $a_i p_i - p'_i$ tends to zero in $\RR^2$, and therefore the continuity of the $G$-action implies that $p_i$ and $p'_i$ have the same limit.

We now turn to~\eqref{thm:af:pt:isom}. We claim that the Dirac mass at the origin is the only $H$-invariant probability measure on $\alpha(\RR^2)$. Let indeed $\mu$ be $H$-invariant; it suffices to show that $\mu$ gives mass zero to the corona. Otherwise, by point~\eqref{thm:af:pt:triv}, we would obtain a $G$-invariant measure on $\alpha(\RR^2)$. This corresponds to a $G$-invariant mean on the space $\ru(G/H)$, which is one of the criteria for $H$ to be co-amenable in $G$, see~\cite[No.~2 \S4]{Eymard72}. However this co-amenability does not hold, because another equivalent criterion (same reference) is the existence of a $G$-invariant mean on $L^\infty(G/H)$, which does not exist due to the phenomenon already mentioned in the proof of Proposition~\ref{prop:D:bnd}.

Now the claim is established and it follows that $\alpha(\RR^2)$ is a $(G,H)$\hyph{}boundary in view of the density of $\RR^2$. To prove that is it the Furstenberg boundary of the pair, it suffices to show that there is a $G$\hyph{}morphism from $\prob(\alpha(\RR^2))$ to $\bnda(G,H)$. Let $x\in \bnda(G,H)$ be an $H$-fixed point. By continuity of the action and compactness, the orbital map of $x$ is a right uniformly continuous map $G/H\to \bnda(G,H)$. Therefore it extends to a continuous $G$-map from $\alpha(\RR^2)$ to $\bnda(G,H)$ and finally using the barycentre map we obtain a $G$\hyph{}morphism from $\prob(\alpha(\RR^2))$ to $\bnda(G,H)$ as desired.

To establish~\eqref{thm:af:pt:quot}, it suffices to show that any net in $\RR^2$ that converges to a point in the corona of $\alpha(\RR^2)$ must converge in direction. Let $s,c\colon \RR^2\to \RR$ be any continuous functions that coincide respectively with the sine and cosine of the direction of vectors in $\RR^2$ outside some compact neighbourhood of the origin. What we have to show is that $s$ and $c$ are in $\ru(G/H)$. This follows from Lemma~\ref{lem:ru}.


Finally, to justify~\eqref{thm:af:pt:huge}, it suffices to exhibit an embedding of $\cbu(\RR_{\geq 0})$ to $\cont(\alpha(\RR^2))$. Here the uniform structure on $\RR_{\geq 0}$ is the usual one (induced from the group $\RR$, or equivalently the usual metric). Given $f\in \cbu(\RR_{\geq 0})$, one checks that the function $\wt{f}$ defined on $v\in \RR^2$ by $\wt{f}(v)= f(\log(\|v\|+1))$ is in $\ru(G/H)$ by using Lemma~\ref{lem:ru}.
\end{proof}

We have now justified all claims made in the Introduction about Example~\ref{exam:affine}. In particular, point~\eqref{thm:af:pt:quot} above implies:

\begin{cor}\label{cor:nomap}
There is no $G$-map $\bnd(G)\to \bnd(G,H)$.\qed
\end{cor}

\begin{rem}\label{rem:no-morph:pairs} 
This example also justifies the claim in Remark~\ref {rem:morph:pairs} that there is in general no $f$-equivariant morphism $\bnd(G', H')\to \bnd(G, H)$ associated to a morphism of pairs $f$. Indeed, in the context of Corollary~\ref{cor:nomap}, we can take $G'=G$ with $f$ the identity and $H'$ trivial.
\end{rem}

We note in passing that $D$ is not a smallest non-trivial $(G,H)$\hyph{}boundary: it admits as a quotient the projective plane $\PP^2$ obtained by identifying each direction with its opposite in the boundary circle of $D$. 

The example of $\PP^2$ would not, however, have allowed us to deduce Corollary~\ref{cor:nomap} since there is an obvious $G$-map $\PP^1\to\PP^2$.

\bigskip
A small variation of Example~\ref{exam:affine}, leaving the connected case, gives a Furstenberg boundary that is a very classical object and in a sense even larger than $\alpha(\RR^2)$ (onto which is naturally projects).

\begin{exam}\label{exam:Samuel}
Let $G= \RR^2 \rtimes \SL_2(\ZZ)$ and $H=\SL_2(\ZZ)$.
%
%
Then $\bnd(G,H) \cong \sigma(\RR^2)$, the Samuel compactification of $\RR^2$. This coincides with the \emph{greatest ambit} of $\RR^2$ and can be defined as the Gelfand spectrum of the algebra of uniformly continuous bounded functions on $\RR^2$.
\end{exam}

The proof is a much simpler version of Theorem~\ref{thm:af} since there is no $H$-continuity to take into account. In fact it is almost the same as Theorem~\ref{thm:malnormal}, only with the topology of $\RR^2$ coming into the picture.

The same arguments show the following. Define $H$ to be $\SL_2^\mathrm{d}(\RR)$, which denotes the group $\SL_2(\RR)$ in the discrete topology. Set $G= \RR^2 \rtimes H$, which is still a locally compact group. Then $\bnd(G,H) \cong \sigma(\RR^2)$, in contrast to Theorem~\ref{thm:af}.


\bibliographystyle{amsplain}
\bibliography{../BIB/ma_bib}

\end{document}